\documentclass[11pt]{article}
\usepackage{amsmath,amssymb,amsthm}
\usepackage{graphicx}
\usepackage{hyperref}
\usepackage{geometry}
\theoremstyle{plain}
\newtheorem{theorem}{Theorem}

\newtheorem{corollary}{Corollary}

\theoremstyle{definition}

\newtheorem{remark}{Remark}

\theoremstyle{remark}

\title{Explicit Robin Green's Functions and Resonance Spectra for the Helmholtz Equation on Balls in All Dimensions}
\author{Ming Yang\\
School of Mathematics, Southeast University \\
Nanjing, 210096 P.R. China \\
\texttt{yangming@seu.edu.cn}\\
}
\date{}

\begin{document}
\maketitle

\begin{abstract}
This work constructs explicit closed-form Green's functions for the Helmholtz equation with Robin boundary conditions on balls in all dimensions $d\ge 2$. Despite the canonical geometry, such kernels have remained unavailable because the Robin condition couples the field and its normal derivative, preventing the method of images and obstructing standard eigenfunction expansions. By a decomposition--expansion method, the free-space fundamental solution is expanded via Graf's addition theorem in two dimensions and the hyperspherical addition theorem in higher dimensions; the regular correction is then determined algebraically by matching Robin boundary data mode by mode. The resulting series converge absolutely on compact interior subsets and are computable to machine precision. These explicit kernels yield a complete characterisation of the resonance spectra: each spectral branch increases strictly with the impedance parameter, interpolating between Neumann and Dirichlet eigenvalues. We establish a low-frequency spectral gap with an explicit cutoff estimate, and a universal high-frequency mode spacing with a second-order correction that distinguishes the Robin condition from the Dirichlet and Neumann extremes. The closed-form expressions furnish exact benchmark solutions for wave simulations in impedance-matched cavities, eliminating geometric discretisation error and providing reference data for the validation of finite-element and boundary-element algorithms in acoustics, electromagnetism and quantum mechanics.
\end{abstract}



\section{Introduction}
\label{sec:intro}

The Green's function for the Helmholtz equation on a ball $\Omega\subset\mathbb{R}^d$ with Robin boundary conditions satisfies
\begin{equation}\label{robin}
\left\{
\begin{array}{ll}
\smallskip\Delta G(x,x_0)+k^2G(x,x_0)=-\delta(x-x_0), & x\in\Omega,\ \ x_0\in\Omega,\\[4pt]
\displaystyle\frac{\partial G}{\partial n} + \alpha G = 0,\quad & x\in\partial\Omega,
\end{array}\right.
\end{equation}
where $k$ is the wave number, $\delta(\cdot)$ is the Dirac delta function, $x_0$ is the fixed source point, $n$ is the unit outer normal to the boundary $\partial\Omega$, and $\alpha>0$ is a constant impedance parameter. This mixed condition arises in acoustics (locally reacting surfaces), electromagnetism (the Leontovich impedance condition for imperfect conductors \cite{SeniorVolakis1995}), and quantum mechanics ($\delta$-shell confinement along cavity boundaries \cite{Albeverio2004}); the limits $\alpha\to0^+$ and $\alpha\to+\infty$ recover the Neumann and Dirichlet extremes, respectively.

\medskip\noindent\textbf{Exact benchmarks and the obstruction to obtaining them.}
Closed-form kernels for canonical geometries are indispensable in computational wave physics: they eliminate geometric discretisation error and provide the only reliable reference solutions against which finite-element or boundary-element algorithms for impedance boundaries can be validated \cite{ColtonKress2019}. The explicit resolvent also encodes the complete resonance spectrum, furnishing non-approximate test data for spectral calibration and Weyl asymptotics in acoustics and electromagnetism \cite{SeniorVolakis1995}. Despite this need, a closed-form Robin Green's function for the Helmholtz equation on the ball has remained unavailable in any dimension $d\ge 2$. The reason is structural: the Robin condition \emph{simultaneously couples} the field and its normal derivative, so the boundary trace of the free-space fundamental solution does \emph{not} decompose into separated scalar data. This coupling prevents the method of images (no elementary system of image sources can satisfy the impedance relation) and obstructs the standard eigenfunction expansion (the Robin trace mixes the two linearly independent radial solutions in a way that does not decouple mode by mode). Existing approaches therefore rely on discretisation, asymptotic expansions, or boundary integral equations; none yields a computable closed-form expression suitable for exact forward models.\medskip

Since Green's 1828 essay \cite{Green1828}, explicit kernels have been central to potential theory and wave propagation. Dirichlet Green's functions for canonical geometries were fully characterised by the early 1900s and underpin the method of images, eigenfunction expansions, and classical Mie scattering theory. The Robin generalisation, however, has resisted explicit solution because of the coupling obstruction described above.

Abstract existence and regularity theory for Robin elliptic systems on general domains has been developed via variational and semigroup methods. Choi and Kim \cite{ChoiKim2014} established well-posedness and Green's function estimates for elliptic and parabolic systems with Robin-type boundary conditions. These results confirm that Robin problems possess solutions, yet they do not produce computable closed-form kernels. For the static Laplace case $k=0$, researchers have derived integral representations on disks and balls using complex-analytic methods and modified Poisson kernels \cite{BegehrVaitekhovich2013, Sadybekov2015, Karachik2019}. However, these constructions rely on the algebraic structure of harmonic functions, including Almansi expansions and the preservation of harmonicity under the Euler operator, and they do not readily extend to the Helmholtz regime. The wavenumber $k$ breaks the core harmonic identities underlying Poisson kernel theory, so the Helmholtz Robin problem has had no explicit solution.

Over the past two decades, the spectral theory of Robin Laplace operators has attracted substantial attention. Daners \cite{Daners2000} gave a general functional-analytic framework for Robin problems on Lipschitz domains. Filonov \cite{Filonov2004} showed that Dirichlet and Neumann eigenvalue sets are strictly separated. Eigenvalue bounds via Bessel quotients were obtained by Freitas \cite{Freitas2021}; see Kennedy \cite{Kennedy2017} for a survey of spectral geometry, monotonicity, and optimisation. Very recently, Ognibene \cite{Ognibene2025} obtained the sharp first-order asymptotics of all Robin eigenvalues in the Dirichlet limit on general Lipschitz domains, identifying a novel geometric quantity as the leading correction term. These results are qualitative or asymptotic. The explicit resolvent kernel for the Robin Helmholtz problem on a ball has not been available; we construct it below. As shown in Section~\ref{sec:spectra}, the poles of this kernel are determined by a single characteristic equation, from which all spectral properties follow.

The one-dimensional case admits an elementary closed form, so we focus on $d\ge 2$. For the disk ($d=2$) and the $d$-dimensional ball ($d\ge 3$), we construct the Robin Green's function by the \emph{decomposition--expansion method}: the Green's function is decomposed into the free-space fundamental solution plus a regular correction; the fundamental solution is expanded on the boundary via Graf's addition theorem (2D) or the hyperspherical addition theorem ($d\ge 3$), the Robin boundary data are matched mode by mode, and the regular correction is determined algebraically. The resulting series are absolutely convergent on compact interior subsets, symmetric in source and observer, and computable to machine precision.

The same kernels also determine the resonance spectrum of the Robin cavity. The resonant wavenumbers are the positive roots $k_{\ell,p}(\alpha)$ of the characteristic equation $\Delta_\ell(k;\alpha)=0$, and each branch varies monotonically with the impedance $\alpha$. Our spectral analysis establishes three quantitative results:
\begin{itemize}
\item[(i)] \textbf{Monotonic spectral interpolation.} For each mode $(\ell,p)$, the resonant wavenumber $k_{\ell,p}(\alpha)$ increases strictly with $\alpha$, interpolating between the Neumann eigenvalue $\mu'_{\ell,p}/R$ (as $\alpha\to0^{+}$) and the Dirichlet eigenvalue $\mu_{\ell,p}/R$ (as $\alpha\to+\infty$).
\item[(ii)] \textbf{Low-frequency spectral gap.} The Robin spectrum exhibits a low-frequency gap: there exists a cutoff $k_c>0$ such that no resonant modes occur for $0<k<k_c$. For $\alpha R\lesssim 1$, we derive the explicit asymptotic estimate
\[
k_c \approx \sqrt{\frac{d\alpha}{R}\left(1-\frac{\alpha R}{d+2}\right)},
\]
and an explicit global approximation for arbitrary $\alpha R$ is given in Section~\ref{sec:kc-estimate}.
\item[(iii)] \textbf{Universal high-frequency mode spacing.} At high frequencies ($p\to\infty$), consecutive resonances satisfy
\[
k_{\ell,p+1}-k_{\ell,p} = \frac{\pi}{R}\left[1-\frac{C_{\ell,d}(\alpha)}{(k_{\ell,p}R)^{2}}+O\bigl((k_{\ell,p}R)^{-3}\bigr)\right],
\]
where
\[
C_{\ell,d}(\alpha)=\alpha R+\ell-\frac{(2\ell+d)^{2}-1}{8}.
\]
The leading-order spacing $\pi/R$ agrees with the Weyl law and is universal, independent of $\alpha$, $\ell$, and $d$, while the second-order correction governed by $C_{\ell,d}(\alpha)$ distinguishes the Robin condition from the Dirichlet and Neumann extremes and encodes both the dimension $d$ and the angular mode $\ell$.
\end{itemize}

The remainder of the paper is structured as follows. Section~\ref{sec:disk} gives the explicit construction on the two-dimensional disk. Section~\ref{sec:ball} carries out the construction for the $d$-dimensional ball with $d\geq 3$. Section~\ref{sec:spectra} carries out the spectral analysis: monotonic eigenvalue behaviour, low-frequency spectral gap, and high-frequency asymptotic spacing. Section~\ref{sec:numerical} presents numerical validation and spectral illustration. Section~\ref{sec:conclusion} summarises the results and indicates directions for future research.

\section{Explicit Construction on the Disk}\label{sec:disk}
Consider the Robin Green's function for the Helmholtz equation on the disk $B_{R}(0)$ of radius $R$:
\begin{equation}\label{eq:robin-disk}
\begin{cases}
\Delta G(x,x_{0}) + k^{2}G(x,x_{0}) = -\delta(x-x_{0}), & x \in B_{R}(0), \quad x_{0} \in B_{R}(0), \\[6pt]
\displaystyle\frac{\partial G(x,x_{0})}{\partial n} + \alpha G(x,x_{0}) = 0, & x \in \partial B_{R}(0).
\end{cases}
\end{equation}

\subsection{Decomposition and boundary data}

We use the canonical decomposition
\begin{equation}\label{eq:decomp-disk}
G(x,x_{0}) = \phi_{2}(x-x_{0}) + v(x),
\end{equation}
where the free-space fundamental solution in two dimensions is
\begin{equation}\label{eq:phi2}
\phi_{2}(x-x_{0}) = \frac{i}{4}H_{0}^{(1)}(k|x-x_{0}|),
\end{equation}
with $H_{n}^{(1)}$ the $n$-th order Hankel function of the first kind. This fundamental solution satisfies
\begin{equation}
\Delta\phi_{2}(x-x_{0}) + k^{2}\phi_{2}(x-x_{0}) = -\delta(x-x_{0}), \quad x \in \mathbb{R}^{2},
\end{equation}
and obeys the two-dimensional Sommerfeld radiation condition. The auxiliary function $v(x)$ satisfies the homogeneous Helmholtz equation in the disk and the nonhomogeneous Robin condition on the boundary:
\begin{equation}\label{eq:v-disk}
\begin{cases}
\Delta v + k^{2}v = 0, & x \in B_{R}(0), \\[6pt]
\mathcal{B}v(x) = -\mathcal{B}\phi_{2}\bigl(x-x_{0}\bigr), & |x| = R,
\end{cases}
\end{equation}
where $\mathcal{B} := \partial_{r} + \alpha$ is the Robin boundary operator at $|x| = R$.

Introduce polar coordinates $x = (r,\theta)$ and $x_{0} = (r_{0},\theta_{0})$. Without loss of generality, set $\theta_{0} = 0$ by rotational symmetry; the general case is restored at the end. Define the boundary data
\begin{equation}\label{eq:gtheta}
g(\theta) := -\mathcal{B}\phi_{2}(x-x_{0})
= -\Bigl(\partial_{r}\phi_{2}(x-x_{0}) + \alpha\phi_{2}(x-x_{0})\Bigr)\bigg|_{r=R}.
\end{equation}

\subsection{Modal expansion via Graf's addition theorem}

The fundamental solution $\phi_{2}$ admits an exact Fourier--Bessel expansion via Graf's addition theorem \cite[\S10.23]{Olver2010}:
\begin{equation}\label{eq:graf}
\phi_{2}(x-x_{0}) = \frac{i}{4}\sum_{n=-\infty}^{\infty} J_{n}(kr_{0}) H_{n}^{(1)}(kr) e^{in\theta}, \quad r > r_{0},
\end{equation}
where $J_{n}$ is the $n$-th order Bessel function of the first kind. This expansion separates the singular behaviour at $r = r_{0}$ (encoded in the Hankel functions) from the regular behaviour at the origin (encoded in the Bessel functions), and is the crucial ingredient that enables explicit coefficient matching.

Substituting \eqref{eq:graf} into the boundary operator yields the modal decomposition of the boundary data:
\begin{equation}\label{eq:g-modal}
g(\theta) = -\frac{i}{4}\sum_{n=-\infty}^{\infty} J_{n}(kr_{0}) \, \mathcal{B}[H_{n}^{(1)}(kr)] \, e^{in\theta},
\end{equation}
where $\mathcal{B}[H_{n}^{(1)}(kr)]$ denotes the Robin boundary trace of the $n$-th Hankel mode.

\subsection{Separation of variables and coefficient matching}

Since $v$ satisfies the homogeneous Helmholtz equation, separation of variables yields the general form
\begin{equation}\label{eq:v-series}
v(r,\theta) = \sum_{n=-\infty}^{\infty} A_{n} J_{n}(kr) e^{in\theta},
\end{equation}
where only Bessel functions $J_{n}$ (regular at the origin) appear; the singular Neumann functions $Y_{n}$ are excluded by the regularity requirement at $r = 0$. Substituting \eqref{eq:v-series} into the boundary condition gives
\begin{equation}\label{eq:v-bc}
g(\theta) = \sum_{n=-\infty}^{\infty} A_{n} \, \mathcal{B}[J_{n}(kr)] \, e^{in\theta},
\end{equation}
where $\mathcal{B}[J_{n}(kr)]$ is the Robin trace of the $n$-th Bessel mode.

Direct computation using the Bessel recurrence relations $J_{n}'(z) = \frac{n}{z}J_{n}(z) - J_{n+1}(z)$ yields the explicit Robin traces:
\begin{align}
\mathcal{B}[J_{n}(kr)] &= \left(\alpha + \frac{n}{R}\right)J_{n}(kR) - kJ_{n+1}(kR), \label{eq:trace-J} \\[6pt]
\mathcal{B}[H_{n}^{(1)}(kr)] &= \left(\alpha + \frac{n}{R}\right)H_{n}^{(1)}(kR) - kH_{n+1}^{(1)}(kR). \label{eq:trace-H}
\end{align}
Matching the Fourier coefficients in \eqref{eq:g-modal} and \eqref{eq:v-bc} gives
\begin{equation}\label{eq:An}
A_{n} = -\frac{i}{4} J_{n}(kr_{0}) \cdot \frac{\mathcal{B}[H_{n}^{(1)}(kr)]}{\mathcal{B}[J_{n}(kr)]}.
\end{equation}

\subsection{Main result}

Restoring the general angle $\theta - \theta_{0}$ by rotational covariance, we obtain the explicit representation.

\begin{theorem}[Robin Green's function on the disk]\label{thm:disk}
The Robin Green's function for the Helmholtz equation on the disk $B_{R}(0)$ is given by
\begin{equation}\label{eq:G-disk}
G(x,x_{0}) =\frac{i}{4}H_{0}^{(1)}(k|x-x_{0}|)
-\frac{i}{4}\sum_{n=-\infty}^{\infty} \frac{\mathcal{B}[H_{n}^{(1)}(kr)]}{\mathcal{B}[J_{n}(kr)]} J_{n}(kr_{0})J_{n}(kr)e^{in(\theta-\theta_{0})},
\end{equation}
where $\mathcal{B}[J_{n}(kr)]$ and $\mathcal{B}[H_{n}^{(1)}(kr)]$ are given by \eqref{eq:trace-J}--\eqref{eq:trace-H}.
\end{theorem}

\begin{proof}
The first term is the free-space fundamental solution $\phi_{2}$ satisfying the singular equation. The second term is the auxiliary function $v$  from \eqref{eq:v-series} and \eqref{eq:An}; it satisfies the homogeneous Helmholtz equation and the Robin boundary condition by construction. Elliptic regularity implies $v$ is real-analytic, hence its Fourier--Bessel series converges absolutely and uniformly on compact subsets of $B_{R}(0)$.
\end{proof}

\section{Explicit Construction on the Ball in \texorpdfstring{$d$}{d} Dimensions (\texorpdfstring{$d\geq 3$}{d >= 3})}\label{sec:ball}

We construct the Robin Green's function for the Helmholtz equation on the $d$-dimensional ball $B_{R}(0)\subset\mathbb{R}^{d}$ of radius $R$, with $d\geq 3$:
\begin{equation}\label{eq:robin-ball}
\begin{cases}
\Delta G(x,x_{0}) + k^{2}G(x,x_{0}) = -\delta(x-x_{0}), & x \in B_{R}(0), \quad x_{0} \in B_{R}(0), \\[6pt]
\displaystyle\frac{\partial G(x,x_{0})}{\partial n} + \alpha G(x,x_{0}) = 0, & x \in \partial B_{R}(0).
\end{cases}
\end{equation}
Set
\begin{equation}\label{eq:nu}
\nu := \frac{d}{2}-1 \geq \frac{1}{2}.
\end{equation}
The free-space fundamental solution in $\mathbb{R}^{d}$ is
\begin{equation}\label{eq:phin}
\phi_{d}(x-x_{0}) = \frac{i}{4}\left(\frac{k}{2\pi|x-x_{0}|}\right)^{\nu} H_{\nu}^{(1)}(k|x-x_{0}|),
\end{equation}
where $H_{\nu}^{(1)}$ is the Hankel function of the first kind of order $\nu$. This satisfies $\Delta\phi_{d}+k^{2}\phi_{d}=-\delta$ in $\mathbb{R}^{d}$ and obeys the Sommerfeld radiation condition.
Decompose the Robin Green's function as
\begin{equation}\label{eq:decomp-ball}
G(x,x_{0}) = \phi_{d}(x-x_{0}) + v(x),
\end{equation}
where the auxiliary function $v$ satisfies the homogeneous Helmholtz equation in the ball and the nonhomogeneous Robin condition on the boundary:
\begin{equation}\label{eq:v-ball}
\begin{cases}
\Delta v + k^{2}v = 0, & x \in B_{R}(0), \\[6pt]
\mathcal{B}v(x) = -\mathcal{B}\phi_{d}(x-x_{0}), & |x| = R,
\end{cases}
\end{equation}
with $\mathcal{B} := \partial_{r} + \alpha$ the Robin boundary operator at $r=|x|=R$.

\subsection{Modal expansion via the hyperspherical addition theorem}

The fundamental solution $\phi_{d}$ admits an exact expansion via the hyperspherical addition theorem for Bessel functions \cite[\S10.23]{Olver2010}. Let $\omega=x/|x|$ and $\omega_{0}=x_{0}/|x_{0}|$ denote points on the unit sphere $\mathbb{S}^{d-1}$, and let $\gamma$ be the angle between $x$ and $x_{0}$, so that $\cos\gamma=\omega\cdot\omega_{0}$. Then for $r>r_{0}$,
\begin{equation}\label{eq:hyperspherical-addition}
\phi_{d}(x-x_{0}) = \frac{i k^{d-2}\Gamma(\nu)}{4\pi^{\nu}} \sum_{\ell=0}^{\infty} (\ell+\nu)\,\frac{J_{\ell+\nu}(kr_{0})}{(kr_{0})^{\nu}}\,
\frac{H_{\ell+\nu}^{(1)}(kr)}{(kr)^{\nu}}\,C_{\ell}^{\nu}
(\cos\gamma),
\end{equation}
where $J_{\mu}$ is the Bessel function of the first kind, $H_{\mu}^{(1)}$ the Hankel function of the first kind, and $C_{\ell}^{\nu}$ the Gegenbauer (ultraspherical) polynomial of degree $\ell$ with index $\nu$.\smallskip

Equivalently, using the orthonormal hyperspherical harmonics $Y_{\ell,m}$ on $\mathbb{S}^{d-1}$, $m=1,\dots,N(d,\ell)$, where $N(d,\ell)=\frac{(2\ell+d-2)\Gamma(\ell+d-2)}{\Gamma(\ell+1)\Gamma(d-1)}$ is the dimension of the space of degree-$\ell$ spherical harmonics in $\mathbb{R}^{d}$, the addition formula reads
\begin{equation}\label{eq:spherical-addition-n}
\phi_{d}(x-x_{0}) = \frac{i\pi k^{d-2}}{2} \sum_{\ell=0}^{\infty}\sum_{m=1}^{N(d,\ell)} \frac{J_{\ell+\nu}(kr_{0})}{(kr_{0})^{\nu}}\,\frac{H_{\ell+\nu}^{(1)}(kr)}{(kr)^{\nu}}\,Y_{\ell,m}(\omega)\,\overline{Y_{\ell,m}}(\omega_{0}), \quad r>r_{0}.
\end{equation}
This separates the singular behaviour at $r=r_{0}$ (encoded in the Hankel functions) from the regular behaviour at the origin (encoded in the Bessel functions) and is the higher-dimensional analog of Graf's expansion \eqref{eq:graf}.

Substituting \eqref{eq:spherical-addition-n} into the boundary operator yields the modal decomposition of the boundary data:
\begin{equation}\label{eq:g-modal-ball}
g(\omega) = -\mathcal{B}\phi_{d}(x-x_{0}) = -\frac{i\pi k^{d-2}}{2}\sum_{\ell=0}^{\infty}\sum_{m=1}^{N(d,\ell)} \mathcal{B}\!\left[\frac{H_{\ell+\nu}^{(1)}(kr)}{(kr)^{\nu}}\right] \frac{J_{\ell+\nu}(kr_{0})}{(kr_{0})^{\nu}}\,Y_{\ell,m}(\omega)\,\overline{Y_{\ell,m}}(\omega_{0}),
\end{equation}
where $\mathcal{B}[\,\cdot\,]$ denotes the Robin boundary trace at $r=R$.

\subsection{Separation of variables and coefficient matching}

Since $v$ satisfies the homogeneous Helmholtz equation, separation of variables in spherical coordinates gives
\begin{equation}\label{eq:v-series-ball}
v(r,\omega) = \sum_{\ell=0}^{\infty}\sum_{m=1}^{N(d,\ell)} B_{\ell}\,\frac{J_{\ell+\nu}(kr)}{(kr)^{\nu}}\,Y_{\ell,m}(\omega)\,\overline{Y_{\ell,m}}(\omega_{0}),
\end{equation}
where only the functions $r^{-\nu}J_{\ell+\nu}(kr)$, regular at the origin, appear; the singular functions $r^{-\nu}Y_{\ell+\nu}(kr)$ are excluded by regularity at $r=0$. The coefficient $B_{\ell}$ is independent of $m$ owing to the rotational symmetry of the Robin boundary condition.

Substituting \eqref{eq:v-series-ball} into the boundary condition gives
\begin{equation}\label{eq:v-bc-ball}
g(\omega) = \sum_{\ell=0}^{\infty}\sum_{m=1}^{N(d,\ell)} B_{\ell}\,\mathcal{B}\!\left[\frac{J_{\ell+\nu}(kr)}{(kr)^{\nu}}\right]\,Y_{\ell,m}(\omega)\,\overline{Y_{\ell,m}}(\omega_{0}).
\end{equation}
Direct computation using the Bessel recurrence $J_{\mu}'(z)=\frac{\mu}{z}J_{\mu}(z)-J_{\mu+1}(z)$ yields the explicit Robin traces:
\begin{align}
\mathcal{B}\!\left[\frac{J_{\ell+\nu}(kr)}{(kr)^{\nu}}\right] &= (kR)^{-\nu}\left[\left(\alpha+\frac{\ell}{R}\right)J_{\ell+\nu}(kR)-kJ_{\ell+\nu+1}(kR)\right], \label{eq:trace-j-n} \\[6pt]
\mathcal{B}\!\left[\frac{H_{\ell+\nu}^{(1)}(kr)}{(kr)^{\nu}}\right] &= (kR)^{-\nu}\left[\left(\alpha+\frac{\ell}{R}\right)H_{\ell+\nu}^{(1)}(kR)-kH_{\ell+\nu+1}^{(1)}(kR)\right]. \label{eq:trace-h-n}
\end{align}
Matching the hyperspherical harmonic coefficients in \eqref{eq:g-modal-ball} and \eqref{eq:v-bc-ball} gives
\begin{equation}\label{eq:Bl}
B_{\ell} = -\frac{i\pi k^{d-2}}{2}\,\frac{J_{\ell+\nu}(kr_{0})}{(kr_{0})^{\nu}}\cdot\frac{\mathcal{B}[H_{\ell+\nu}^{(1)}(kr)/(kr)^{\nu}]}{\mathcal{B}[J_{\ell+\nu}(kr)/(kr)^{\nu}]}.
\end{equation}

\subsection{Main result}

\begin{theorem}[Robin Green's function on the $d$-ball]\label{thm:ball}
Let $d\geq 3$ and $\nu=d/2-1$. The Robin Green's function for the Helmholtz equation on the ball $B_{R}(0)\subset\mathbb{R}^{d}$ is given by
\begin{multline}\label{eq:G-ball}
G(x,x_{0}) = \frac{i}{4}\left(\frac{k}{2\pi|x-x_{0}|}\right)^{\nu} H_{\nu}^{(1)}(k|x-x_{0}|) \\
-\frac{i\pi k^{d-2}}{2}\sum_{\ell=0}^{\infty}\sum_{m=1}^{N(d,\ell)} \frac{\mathcal{B}[H_{\ell+\nu}^{(1)}(kr)/(kr)^{\nu}]}{\mathcal{B}[J_{\ell+\nu}(kr)/(kr)^{\nu}]}\,\frac{J_{\ell+\nu}(kr_{0})}{(kr_{0})^{\nu}}\,\frac{J_{\ell+\nu}(kr)}{(kr)^{\nu}}
\,Y_{\ell,m}(\omega)\,\overline{Y_{\ell,m}}(\omega_{0}),\quad
\end{multline}
where $\mathcal{B}[\,\cdot\,]$ is given by \eqref{eq:trace-j-n}--\eqref{eq:trace-h-n}.
\end{theorem}

\begin{proof}
The first term is the free-space fundamental solution $\phi_{d}$ satisfying the singular equation. The second term is the auxiliary function $v$ from \eqref{eq:v-series-ball} and \eqref{eq:Bl}; it satisfies the homogeneous Helmholtz equation and the Robin boundary condition by construction. Elliptic regularity implies $v$ is real-analytic, hence its hyperspherical harmonic series converges absolutely and uniformly on compact subsets of $B_{R}(0)$.
\end{proof}

\begin{corollary}[Three-dimensional ball]\label{cor:3d}
For $d=3$ we have $\nu=1/2$, and the half-integer Bessel functions reduce to the spherical Bessel functions:
\[
j_{\ell}(z)=\sqrt{\frac{\pi}{2z}}\,J_{\ell+1/2}(z),\qquad
h_{\ell}^{(1)}(z)=\sqrt{\frac{\pi}{2z}}\,H_{\ell+1/2}^{(1)}(z).
\]
Consequently \eqref{eq:G-ball} becomes
\begin{equation}\label{eq:G-3d}
G(x,x_{0})=\frac{e^{ik|x-x_{0}|}}{4\pi|x-x_{0}|}-ik\sum_{\ell=0}^{\infty}\sum_{m=1}^{2\ell+1}
\frac{\mathcal{B}[h_{\ell}^{(1)}(kr)]}{\mathcal{B}[j_{\ell}(kr)]}\,
j_{\ell}(kr_{0})\,j_{\ell}(kr)\,
Y_{\ell,m}(\omega)\,\overline{Y_{\ell,m}}(\omega_{0}),
\end{equation}
with
\[
\mathcal{B}[j_{\ell}(kr)]=\left(\alpha+\frac{\ell}{R}\right)j_{\ell}(kR)-kj_{\ell+1}(kR),
\]
and analogously for $h_{\ell}^{(1)}$.  This is the classical Mie-type expansion for the impedance sphere (Leontovich boundary).
\end{corollary}

\begin{remark}[Dimensional analogy and Mie theory]\label{rem:analogy}
The formulas for dimension $d=2$ \eqref{eq:G-disk} and arbitrary $d\geq 3$ \eqref{eq:G-ball} share a common algebraic structure. In every dimension the Robin trace takes the form
\[
\mathcal{B}[r^{-\nu}Z_{\ell+\nu}(kr)] = (kR)^{-\nu}\left[\left(\alpha+\frac{\ell}{R}\right)Z_{\ell+\nu}(kR)-kZ_{\ell+\nu+1}(kR)\right],
\]
where $Z_{\mu}$ stands for either $J_{\mu}$ or $H_{\mu}^{(1)}$ and $\nu=0$ in 2D (since $J_{\ell}(z)/z^{0}=J_{\ell}(z)$) while $\nu=d/2-1$ in dimension $d$. This reflects the unified radial ODE
\[
u''+\frac{d-1}{r}u'+\left(k^{2}-\frac{\ell(\ell+d-2)}{r^{2}}\right)u=0,
\]
whose regular solution at the origin is $r^{-\nu}J_{\ell+\nu}(kr)$.

In the Dirichlet limit $\alpha\to+\infty$, the Robin condition reduces to the perfectly conducting boundary condition, and for $d=3$ the coefficients \eqref{eq:Bl} formally resemble the classical Mie scattering coefficients. For finite $\alpha>0$, the Robin generalisation yields the exact solution for impedance spheres under the Leontovich boundary condition, the standard model in electromagnetic and acoustic scattering by lossy or coated particles \cite{SeniorVolakis1995}.
\end{remark}

\section{Resonance Spectra}
\label{sec:spectra}

To treat the disk ($d=2$) and the $d$-ball ($d\geq3$) in a unified way, we introduce
\begin{equation}
\label{eq:nu-def}
\nu:=\frac{d}{2}-1\geq0,
\end{equation}
and let $\ell=0,1,2,\ldots$ denote the angular quantum number (in $2$D this is the Fourier mode index $n$ of Section~\ref{sec:disk}; in $d\geq3$ it is the degree of the hyperspherical harmonics of Section~\ref{sec:ball}).  The radial Bessel order is $\mu=\ell+\nu$.  From Sections~\ref{sec:disk} and~\ref{sec:ball}, the characteristic equation governing the poles of the Green's function is
\begin{equation}
\label{eq:char-unified}
\Delta_{\ell}(k;\alpha)=\Bigl(\alpha+\frac{\ell}{R}\Bigr)J_{\mu}(kR)-kJ_{\mu+1}(kR)=0,
\qquad\mu=\ell+\nu,
\end{equation}
or equivalently, using $J_{\mu}'(z)=\frac{\mu}{z}J_{\mu}(z)-J_{\mu+1}(z)$,
\begin{equation}
\label{eq:char-alt}
kJ_{\mu}'(kR)+\Bigl(\alpha-\frac{\nu}{R}\Bigr)J_{\mu}(kR)=0.
\end{equation}
The positive roots $k_{\ell,p}(\alpha)$, $p=1,2,\ldots$, are the resonant wavenumbers of the Robin cavity.

\begin{remark}[Universality of the characteristic equation]
\label{rem:universal}
The explicit characteristic determinant $\Delta_{\ell}(k;\alpha)$
in~\eqref{eq:char-unified} encodes the complete spectral structure
of the Robin problem on the ball: the location of all resonant
wavenumbers, their analyticity and strict monotonicity in $\alpha$,
the interpolation between Neumann and Dirichlet eigenvalues, the
low-frequency spectral gap, the high-frequency mode spacing, and
the full Neumann-Dirichlet limit asymptotic expansion.  Once the explicit
form of $\Delta_{\ell}(k;\alpha)$ is available, every quantitative
spectral property can be extracted by elementary Bessel analysis,
without recourse to abstract operator theory or Dirichlet-to-Neumann
duality.  This universality provides a computable benchmark against which general-domain results~\cite{Daners2000,Filonov2004,
Freitas2021,Kennedy2017,Ognibene2025} can be tested and calibrated.
\end{remark}

\subsection{Fundamental spectral properties}
\label{subsec:fundamental}

\begin{theorem}
\label{thm:spectral-basic}
Let $d\geq2$, $\nu=d/2-1$, $\ell\geq0$, and $\alpha>0$.  Then
\begin{enumerate}
\item
\label{item:discrete}
\textbf{Discreteness and reality.}  For each fixed $\ell$, the roots $k_{\ell,p}(\alpha)$, $p=1,2,\ldots$, are positive, real, form an increasing sequence with $k_{\ell,p}\to+\infty$, and each corresponding radial eigenfunction is unique up to a constant multiple.
\item
\label{item:analytic}
\textbf{Analyticity.}  The characteristic determinant $\Delta_{\ell}(k;\alpha)$ is real-analytic in $(k,\alpha)\in(0,+\infty)^{2}$.  Consequently, each branch $k_{\ell,p}(\alpha)$ is real-analytic in $\alpha$.
\item
\label{item:monotone}
\textbf{Strict monotonicity.}  For each fixed $\ell$ and $p$, $k_{\ell,p}(\alpha)$ is strictly monotonically increasing in $\alpha\in(0,+\infty)$.
\item
\label{item:interp}
\textbf{Interpolation and bounds.}  As $\alpha\to0^{+}$,
\begin{equation}
\label{eq:Neumann-limit}
k_{\ell,p}(\alpha)\to\frac{\mu_{\ell,p}'}{R},
\end{equation}
where $\mu_{\ell,p}'$ is the $p$-th positive root of
\begin{equation}
\ell J_{\mu}(x)-xJ_{\mu+1}(x)=0
\quad\bigl(\text{equivalently }J_{\mu}'(x)=\tfrac{\nu}{x}J_{\mu}(x)\bigr),
\end{equation}
with the convention $\mu_{0,1}'=0$ for all $d\geq2$ (corresponding to the constant eigenfunction).  As $\alpha\to+\infty$,
\begin{equation}
\label{eq:Dirichlet-limit}
k_{\ell,p}(\alpha)\to\frac{\mu_{\ell,p}}{R},
\end{equation}
where $\mu_{\ell,p}$ is the $p$-th positive zero of $J_{\mu}(x)$.  Moreover,
\begin{equation}
\label{eq:bounds}
\frac{\mu_{\ell,p}'}{R}<k_{\ell,p}(\alpha)<\frac{\mu_{\ell,p}}{R},
\qquad\forall\;\alpha\in(0,+\infty).
\end{equation}
\end{enumerate}
\end{theorem}

\begin{proof}
(i) Equation~\eqref{eq:char-unified} is the spectral condition for the self-adjoint Sturm--Liouville problem
\begin{equation}
\label{eq:SL}
\begin{cases}
u''+\dfrac{d-1}{r}u'+\Bigl(\lambda-\dfrac{\ell(\ell+d-2)}{r^{2}}\Bigr)u=0, & 0<r<R,\\[6pt]
|u(0)|<+\infty,\quad u'(R)+\alpha u(R)=0,
\end{cases}
\end{equation}
 with $\lambda=k^{2}$.  Standard spectral theory for singular self-adjoint Sturm--Liouville operators on $(0,R)$ yields the result.\smallskip

(ii) The Bessel functions $J_{\mu}(z)$ are entire, hence $\Delta_{\ell}(k;\alpha)$ is real-analytic.  By (i), each root is simple, so $\partial_{k}\Delta_{\ell}|_{k=k_{\ell,p}}\neq0$.  The implicit function theorem for real-analytic functions guarantees local analyticity; global analyticity follows from simple connectedness and root separation.\smallskip

(iii) Consider the energy functional for the Robin Laplacian on $\overline{B_{R}(0)}$ restricted to the angular sector indexed by $\ell$:
\begin{equation}
I[u;\alpha]=\int_{0}^{R}\Bigl(|u'(r)|^{2}+\frac{\ell(\ell+d-2)}{r^{2}}|u(r)|^{2}\Bigr)r^{d-1}\,dr+\alpha R^{d-1}|u(R)|^{2},
\end{equation}
where $u$ ranges over all non-zero radial functions in $H^{1}(0,R)$ with $|u(0)|<+\infty$
(the first eigenvalue is the infimum of the Rayleigh quotient, the higher ones are given by the min--max principle).  Let $u_{\ell,p}$ be the normalised eigenfunction corresponding to $\lambda_{\ell,p}=k_{\ell,p}^{2}$, which is a critical point of $I$ subject to $\|u\|_{L^{2}(0,R;r^{d-1}dr)}=1$.  Then $\lambda_{\ell,p}=I[u_{\ell,p};\alpha]$.  Differentiating with respect to $\alpha$:
\begin{equation}
\frac{\partial\lambda_{\ell,p}}{\partial\alpha}=\frac{\partial I}{\partial\alpha}\Big|_{u=u_{\ell,p}}+\Bigl\langle\frac{\delta I}{\delta u}\Big|_{u=u_{\ell,p}},\frac{\partial u_{\ell,p}}{\partial\alpha}\Bigr\rangle.
\end{equation}
Since $u_{\ell,p}$ is a critical point of $I$ subject to the normalisation constraint, 
the first variation of $I$ at $u_{\ell,p}$ is proportional to the variation of the constraint, 
whose pairing with $\partial_{\alpha}u_{\ell,p}$ vanishes because 
$\frac{d}{d\alpha}\|u_{\ell,p}\|^{2}=0$ (the envelope, or Hellmann--Feynman, argument).  Moreover,
\begin{equation}
\frac{\partial I}{\partial\alpha}\Big|_{u=u_{\ell,p}}=R^{d-1}|u_{\ell,p}(R)|^{2}>0,
\end{equation}
because the eigenfunction $u_{\ell,p}$ is non-trivial and therefore does not vanish identically on the boundary.  Hence
\begin{equation}
\frac{\partial\lambda_{\ell,p}}{\partial\alpha}=R^{d-1}|u_{\ell,p}(R)|^{2}>0.
\end{equation}
Since $\lambda_{\ell,p}=k_{\ell,p}^{2}$ and $k_{\ell,p}>0$, it follows that
\begin{equation}
\frac{\partial k_{\ell,p}}{\partial\alpha}=\frac{1}{2k_{\ell,p}}\frac{\partial\lambda_{\ell,p}}{\partial\alpha}>0,
\end{equation}
and $k_{\ell,p}(\alpha)$ is strictly increasing on $(0,+\infty)$.

(iv) The limits follow by letting $\alpha\to0$ and $\alpha\to+\infty$ in~\eqref{eq:char-unified}.  The bounds~\eqref{eq:bounds} follow from strict monotonicity.
\end{proof}

\begin{remark}[Explicit impedance formula]
\label{rem:inverse}
Strict monotonicity of each spectral branch yields a bijection $\alpha\mapsto k_{\ell,p}(\alpha)$ onto $(\mu_{\ell,p}'/R,\mu_{\ell,p}/R)$.  Consequently, for any measured resonant wavenumber $k$ lying strictly between the Neumann and Dirichlet eigenvalues of a known mode $(\ell,p)$, the impedance is recovered explicitly by
\begin{equation}
\label{eq:impedance-inverse}
\alpha=\frac{kJ_{\mu+1}(kR)}{J_{\mu}(kR)}-\frac{\ell}{R},\qquad\mu=\ell+\nu.
\end{equation}
\end{remark}

\begin{remark}[Geometry and physics of the spectral interpolation]
\label{rem:interp-physics}
The monotonicity $\mathrm{d}k_{\ell,p}/\mathrm{d}\alpha>0$ shows that the Robin spectrum interpolates continuously between the Neumann and Dirichlet extremes, with each mode $(\ell,p)$  evolving independently and without crossing. This interpolation is sharp: the bounds~\eqref{eq:bounds} are uniform in $\alpha$ and the endpoints are attained only in the limits. Filonov~\cite{Filonov2004} proved strict separation for general Lipschitz domains, ensuring non-zero width.
When $\alpha\to0^{+}$ one recovers the Neumann (hard) limit, and when $\alpha\to+\infty$ the Dirichlet (soft) limit; intermediate values of $\alpha$ correspond to impedance-matched cavities.  In quantum mechanics, $\alpha$ is the $\delta$-shell strength; our formulas give the exact resolvent for partial confinement~\cite{Albeverio2004}.
\end{remark}

\begin{remark}[Neumann--Dirichlet limiting asymptotics]
\label{rem:limits}
The explicit characteristic equation~\eqref{eq:char-unified} yields the complete
asymptotic expansions of $k_{\ell,p}(\alpha)$ at both extremes of the impedance.

\medskip\noindent\textbf{(i) Neumann limit ($\alpha\to0^{+}$).}
For modes with $\mu_{\ell,p}'>0$ (i.e. $\ell\ge1$ or $p\ge2$),
\begin{equation}
\label{eq:Neumann-asymp}
k_{\ell,p}(\alpha)=\frac{\mu_{\ell,p}'}{R}
+\frac{\mu_{\ell,p}'}{(\mu_{\ell,p}')^{2}-\ell(\ell+d-2)}\,
\alpha+O(\alpha^{2}),\qquad\alpha\to0^{+}.
\end{equation}
For the radially symmetric ground mode $(\ell,p)=(0,1)$, where $\mu_{0,1}'=0$,
the characteristic equation gives instead
\begin{equation}
\label{eq:Neumann-asymp-00}
k_{0,1}(\alpha)=\frac{\sqrt{d\alpha}}{R}+O(\alpha^{3/2}),\qquad\alpha\to0^{+}.
\end{equation}

\medskip\noindent\textbf{(ii) Dirichlet limit ($\alpha\to+\infty$).}
\begin{equation}
\label{eq:Dirichlet-asymp}
k_{\ell,p}(\alpha)=\frac{\mu_{\ell,p}}{R}-\frac{\mu_{\ell,p}}{\alpha R^{2}}
-\frac{(d-3)\,\mu_{\ell,p}}{2\alpha^{2}R^{3}}+O\!\left(\frac{1}{\alpha^{3}}\right),
\qquad\alpha\to+\infty.
\end{equation}
For $d=3$ the $O(1/\alpha^{2})$
term vanishes identically, a dimension-specific phenomenon with no counterpart
in the general Lipschitz theory~\cite{Ognibene2025}.
\end{remark}

\subsection{Low-frequency spectral gap}
\label{subsec:gap}

\begin{theorem}
\label{thm:gap}
Let $d\geq2$, $\nu=d/2-1$, $\ell\geq0$, and $\alpha>0$.  Then
\begin{enumerate}
\item
\textbf{Asymptotics as $k\to0^{+}$.}  For the radially symmetric mode $\ell=0$,
\begin{equation}
\label{eq:Delta0-smallk}
\Delta_{0}(k)=\frac{(kR/2)^{\nu}}{\Gamma(\nu+1)}\Bigl[\alpha-\frac{\alpha+\frac{2}{R}}{4(\nu+1)}(kR)^{2}
+\frac{\alpha R+4}{32R(\nu+1)(\nu+2)}(kR)^{4}+O(k^{6})\Bigr].
\end{equation}
For $\ell\geq1$,
\begin{equation}
\label{eq:Deltal-smallk}
\Delta_{\ell}(k)=\frac{R^{\mu}}{2^{\mu}\Gamma(\mu+1)}\Bigl(\alpha+\frac{\ell}{R}\Bigr)k^{\mu}+O(k^{\mu+2}),\qquad\mu=\ell+\nu.
\end{equation}
\item
\textbf{Low-frequency spectral gap.}  There exists $k_{c}>0$ such that $\Delta_{\ell}(k)>0$ for all $0<k<k_{c}$ and all $\ell\geq0$.  Consequently, the Robin Laplacian has no eigenvalues in $(0,k_{c}^{2})$; equivalently, no resonant modes exist in the wavenumber band $(0,k_{c})$.
\item
\textbf{Small-impedance estimate for the critical wavenumber.}  As $\beta=\alpha R\to0^{+}$,
\begin{equation}
\label{eq:kc-series-gap}
x_{c}^{2}=d\beta-\frac{d}{d+2}\,\beta^{2}+O(\beta^{3}),\qquad x_{c}=k_{c}R,
\end{equation}
and hence, for $\alpha R\lesssim1$,
\begin{equation}
\label{eq:kc-estimate}
k_{c}\approx\sqrt{\frac{d\alpha}{R}\left(1-\frac{\alpha R}{d+2}\right)}.
\end{equation}
\end{enumerate}
\end{theorem}

\begin{proof}
(i) \eqref{eq:Delta0-smallk} and~\eqref{eq:Deltal-smallk} follow from the small-argument expansion of $J_{\mu}(z)$.
(ii) The asymptotics show $\Delta_{\ell}(k)>0$ near $k=0$ while Theorem~\ref{thm:spectral-basic} guarantees eventual sign changes.  The first positive root $k_{c,\ell}$ exists for each $\ell$.  By the Sturm--Liouville variational principle, increasing $\ell$ raises the effective centrifugal barrier $\ell(\ell+d-2)/r^{2}$, hence $k_{c,0}<k_{c,1}<k_{c,2}<\cdots$.  Thus $k_{c}:=\min_{\ell\geq0}k_{c,\ell}=k_{c,0}>0$. (iii) Write $\Delta_{0}(k;\alpha)=0$ as $xJ_{\nu+1}(x)=\beta J_{\nu}(x)$ with $x=kR$ and $\beta=\alpha R$.  Inserting the power series of $J_{\nu}$ and $J_{\nu+1}$ and solving for $t=x^{2}$ as a series in $\beta$ gives $t=2(\nu+1)\beta-\frac{\nu+1}{\nu+2}\beta^{2}+O(\beta^{3})$, 
which is~\eqref{eq:kc-series-gap} since $2(\nu+1)=d$ and $\frac{\nu+1}{\nu+2}=\frac{d}{d+2}$; the estimate~\eqref{eq:kc-estimate} follows by truncation.
\end{proof}

\begin{remark}
The global critical wavenumber $k_{c}=k_{c,0}$ is determined by the radially symmetric $\ell=0$ mode, which has the lowest resonance threshold.  This agrees with the physical intuition that $s$-wave (angularly independent) modes penetrate most easily through the impedance boundary.  Numerically, $k_{c}$ is the first positive root of $\Delta_{0}(k;\alpha)=0$, found by standard root-finding algorithms (e.g., Brent's method).  The explicit estimate~\eqref{eq:kc-estimate} approximates this root accurately when $\alpha R\lesssim1$.  An explicit global approximation of $k_{c}$ for arbitrary $\alpha R$ is given in Section~\ref{sec:kc-estimate}, whose relative error stays below $1.7\%$ on the whole impedance range for $d=2,3,4$; see Table~\ref{tab:kc-global}.
\end{remark}

\begin{remark}[Comparison with prior literature]
The positivity of the first Robin eigenvalue for $\alpha>0$ is a classical consequence of the Rayleigh quotient and the Sturm--Liouville framework, see e.g.~Daners~\cite{Daners2000} and Freitas~\cite{Freitas2021}.  What Theorem~\ref{thm:gap} adds is a quantitative, explicit asymptotic estimate~\eqref{eq:kc-estimate} for the size of the low-frequency spectral gap, obtained directly from the small-$k$ expansion of the characteristic determinant and valid in every dimension $d\geq2$.
\end{remark}

\subsection{High-frequency asymptotic resonance spacing}
\label{subsec:high-freq}

\begin{theorem}
\label{thm:spacing}
Let $d\geq2$, $\nu=d/2-1$, $\ell\geq0$, and $\alpha>0$.  Then
\begin{enumerate}
\item
\textbf{Asymptotics of the characteristic determinant as $k\to+\infty$.}
\begin{equation}
\label{eq:Delta-largek}
\Delta_{\ell}(k)=-\sqrt{\frac{2k}{\pi R}}\,
\Bigl[\sin\theta(kR)-\frac{C_{\ell,d}(\alpha)}{kR}\cos\theta(kR)+O\bigl((kR)^{-2}\bigr)\Bigr],
\qquad k\to+\infty,
\end{equation}
where
\begin{equation}
\label{eq:theta-def}
\theta(z):=z-\frac{\mu\pi}{2}-\frac{\pi}{4},\qquad\mu=\ell+\nu,
\end{equation}
and
\begin{equation}
\label{eq:C-def}
C_{\ell,d}(\alpha):=\alpha R+\ell-\frac{(2\ell+d)^{2}-1}{8}.
\end{equation}
\item
\textbf{Universal spacing with second-order correction.}  As $p\to\infty$,
\begin{equation}
\label{eq:spacing}
k_{\ell,p+1}-k_{\ell,p}=\frac{\pi}{R}\Bigl[1-\frac{C_{\ell,d}(\alpha)}{(k_{\ell,p}R)^{2}}+O\bigl((k_{\ell,p}R)^{-3}\bigr)\Bigr].
\end{equation}
\end{enumerate}
The leading-order term $\pi/R$ is independent of $\alpha$, $\ell$, and the boundary condition type: it coincides with the classical spacing of consecutive Bessel zeros given by the McMahon expansion \cite[\S10.21]{Olver2010}, and is consistent with the Weyl asymptotics for the Laplacian on a ball.  The second-order correction, governed by $C_{\ell,d}(\alpha)$, distinguishes the Robin condition from the Dirichlet and Neumann extremes and encodes both the dimension $d$ and the angular mode $\ell$.
\end{theorem}

\begin{proof}
(i)\  We start from the characteristic equation in the form~
\eqref{eq:char-unified},
\begin{equation}
\label{eq:char-proof}
\Delta_{\ell}(k)=\Bigl(\alpha+\frac{\ell}{R}\Bigr)J_{\mu}(kR)-kJ_{\mu+1}(kR),
\qquad\mu=\ell+\nu,
\end{equation}
and employ the large-argument expansions of the Bessel functions~
\cite[\S10.17]{Olver2010}:
\begin{align}
J_{\mu}(z)&=\sqrt{\frac{2}{\pi z}}
\Bigl[\cos\theta(z)-\frac{4\mu^{2}-1}{8z}\sin\theta(z)+O(z^{-2})\Bigr],
\\
J_{\mu+1}(z)&=\sqrt{\frac{2}{\pi z}}
\Bigl[\sin\theta(z)+\frac{4(\mu+1)^{2}-1}{8z}\cos\theta(z)+O(z^{-2})\Bigr],
\end{align}
where we have used $\theta_{\mu+1}(z)=\theta(z)-\pi/2$ in the second line.
Substituting these into~
\eqref{eq:char-proof}, the term of order $(kR)^{-2}$ dropped from the first bracket is $(\alpha R+\ell)\frac{4\mu^{2}-1}{8}(kR)^{-2}\sin\theta(kR)$, which is absorbed by the remainder, and we obtain
\begin{equation}
\Delta_{\ell}(k)=-\sqrt{\frac{2k}{\pi R}}
\Bigl[\sin\theta(kR)
-\frac{\alpha R+\ell-\dfrac{(2\mu+2)^{2}-1}{8}}{kR}\cos\theta(kR)
+O\bigl((kR)^{-2}\bigr)\Bigr].
\end{equation}
Since $\mu=\ell+\nu$ and $\nu=d/2-1$, we have $2\mu+2=2\ell+d$,
whence
\begin{equation}
\alpha R+\ell-\frac{(2\mu+2)^{2}-1}{8}
=\alpha R+\ell-\frac{(2\ell+d)^{2}-1}{8}
=:C_{\ell,d}(\alpha),
\end{equation}
which establishes~
\eqref{eq:Delta-largek}.

(ii)\  The resonant wavenumbers satisfy $\Delta_{\ell}(k_{\ell,p})=0$.
Since $\cos\theta(k_{\ell,p}R)=\pm1+O\bigl((k_{\ell,p}R)^{-2}\bigr)$ at the roots,
division by $\cos\theta(kR)$ in~
\eqref{eq:Delta-largek} is legitimate, and the roots are asymptotic solutions of
\begin{equation}
\tan\theta(kR)=\frac{C_{\ell,d}(\alpha)}{kR}+O\bigl((kR)^{-2}\bigr).
\end{equation}
Set $x_{p}:=k_{\ell,p}R$.  The solutions of this equation satisfy, for some fixed integer $m_{0}$ depending only on the labeling of the roots,
\begin{equation}
\theta(x_{p})=(p+m_{0})\pi+\arctan\Bigl(\frac{C_{\ell,d}(\alpha)}{x_{p}}\Bigr)
+O(x_{p}^{-2})
=(p+m_{0})\pi+\frac{C_{\ell,d}(\alpha)}{x_{p}}+O(x_{p}^{-2}).
\end{equation}
The fixed shift $m_{0}$ is absorbed by relabelling $p\mapsto p+m_{0}$, which does not affect the spacing; with this convention, expanding $\theta(x)=x-\frac{(2\mu+1)\pi}{4}$ gives
\begin{equation}
x_{p}=p\pi+\frac{(2\mu+1)\pi}{4}+\frac{C_{\ell,d}(\alpha)}{p\pi}+O(p^{-2}).
\end{equation}
Hence
\begin{equation}
x_{p+1}-x_{p}=\pi-\frac{C_{\ell,d}(\alpha)}{\pi p^{2}}+O(p^{-3}).
\end{equation}
Since $x_{p}=p\pi+O(1)$, we have $p^{-2}=\pi^{2}x_{p}^{-2}+O(p^{-3})$,
so replacing $p$ by $x_{p}/\pi$ on the right-hand side introduces only
$O(x_{p}^{-3})$ terms, giving
\begin{equation}
x_{p+1}-x_{p}=\pi\Bigl[1-\frac{C_{\ell,d}(\alpha)}{x_{p}^{2}}
+O(x_{p}^{-3})\Bigr].
\end{equation}
Dividing by $R$ yields~
\eqref{eq:spacing}.
\end{proof}

\begin{remark}[Exactness in three dimensions]
\label{rem:exact-3d}
For $d=3$ ($\nu=\tfrac12$) the Bessel orders are half-integers and the asymptotic series of \cite[\S10.17]{Olver2010} terminate; the expansion~\eqref{eq:Delta-largek} is then an \emph{exact} identity.  For instance, for the radially symmetric mode $\ell=0$ one has
\begin{equation}
\Delta_{0}(k)=\sqrt{\frac{2k}{\pi R}}\Bigl[\cos(kR)+\frac{\alpha R-1}{kR}\sin(kR)\Bigr],
\end{equation}
so the resonant wavenumbers are the exact roots of $\tan(kR)=-\dfrac{kR}{\alpha R-1}$, and $C_{0,3}(\alpha)=\alpha R-1$.  In particular, at $\alpha R=1$ the consecutive roots are exactly $(p-\tfrac12)\pi/R$, so the spacing~\eqref{eq:spacing} holds with vanishing correction for all $p$.
\end{remark}

\section{Numerical Validation and Spectral Illustration}\label{sec:numerical}

In this section we verify the explicit series representations numerically and validate the spectral predictions of Theorems~\ref{thm:spectral-basic}--\ref{thm:spacing}.  All computations are performed on the unit disk $B_{1}(0)\subset\mathbb{R}^{2}$ and the unit ball $B_{1}(0)\subset\mathbb{R}^{3}$, with wavenumber $k=3.0$ and impedance $\alpha=1.0$ unless stated otherwise.  The series representations are symmetric, converge absolutely on compact interior subsets, and can be evaluated to high accuracy in any dimension $d\geq 2$.

\subsection{Field structure and convergence}
Since $G=\phi+v$ by construction, it suffices to validate the regularity, homogeneous equation, and Robin boundary condition of the smooth auxiliary function $v$. 
Figure~\ref{fig:field} shows the magnitude of the Robin Green's function on the unit disk ($d=2$) and on the xz-plane cross section of the unit ball ($d=3$), together with the exponential convergence of the modal series.  The source point $x_{0}$ is marked by the red star.  The field exhibits the expected monopole-like singular pattern near the source, with Robin impedance producing a characteristic phase shift at the boundary.  The modal series converges exponentially in both dimensions, reaching machine precision ($\sim 10^{-16}$) at $N\approx 20$.  Algebraic verification of the boundary condition $\mathcal{B}v+\mathcal{B}\phi=0$ confirms the correctness of the modal matching to machine precision.

\begin{figure}[htbp]
\centering
\includegraphics[width=0.95\textwidth]{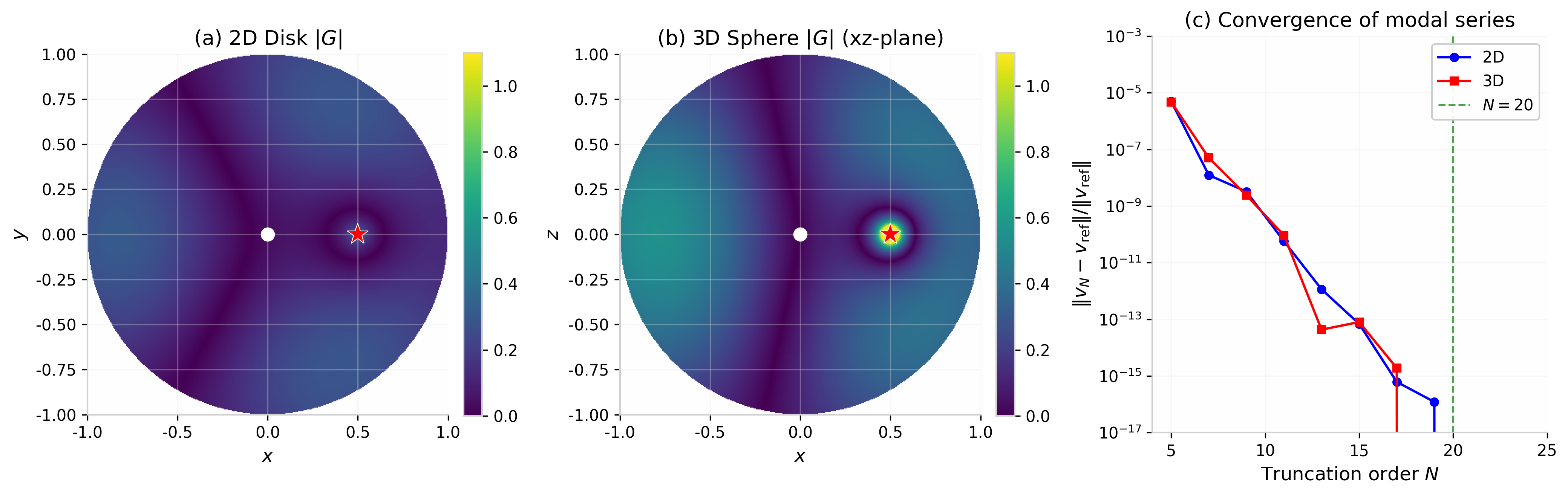}
\caption{(a)~Magnitude of the Robin Green's function on the unit disk ($d=2$).  (b)~Magnitude on the xz-plane cross section of the unit ball ($d=3$).  (c)~Exponential convergence of the modal series: relative error of $v$ versus truncation order $N$ for $k=3$, $R=1$, $\alpha=1$, $r_0=0.5$.  Machine precision is reached at $N\approx 20$ in both dimensions.}
\label{fig:field}
\end{figure}

\subsection{High-wavenumber robustness and modal scaling}

To test short-wavelength behaviour, we take $k=30$ ($\lambda\approx 0.21$) on the unit disk ($d=2$).  This requires resolving nearly ten wavelengths across the diameter.  At $k=30$, $N=20$ is insufficient: the error stagnates and the boundary residual degrades to $10^{-3}$.  This reflects the physical requirement that the number of angular modes must satisfy $N\gtrsim kR$ to resolve the oscillatory Bessel functions.  Table~\ref{tab:high-k} confirms this scaling by increasing $N$ systematically.

\begin{table}[htbp]
\centering
\caption{Convergence of the explicit series for $d=2$, $k=30$ ($\lambda\approx 0.21$) with increasing truncation order $N$ ($\alpha=1.0$, $r_{0}=0.5$, $N_{\text{ref}}=50$)}\label{tab:high-k}
\vspace{0.3cm}
\begin{tabular}{lccc}
\hline
$N$ & $\|v_{N}-v_{\text{ref}}\|/\|v_{\text{ref}}\|$ & Residual $L^{\infty}$ & Residual $L^{2}$ \\
\hline
20 & $3.9\times10^{-2}$ & $7.4\times10^{-3}$ & $4.3\times10^{-3}$ \\
30 & $8.4\times10^{-9}$ & $5.3\times10^{-8}$ & $2.6\times10^{-8}$ \\
35 & $5.1\times10^{-11}$ & $6.3\times10^{-10}$ & $2.8\times10^{-10}$ \\
40 & $5.9\times10^{-13}$ & $1.0\times10^{-11}$ & $4.4\times10^{-12}$ \\
50 & $0$ & $1.1\times10^{-14}$ & $3.8\times10^{-15}$ \\
\hline
\end{tabular}
\end{table}

The error drops by seven orders of magnitude when $N$ reaches $30\approx kR$, and machine precision is attained at $N=40$.  This scaling reflects the Nyquist--Shannon sampling theorem in angular frequency space: $kR$ modes are needed to resolve the highest angular frequency supported by the oscillatory Bessel functions.  Unlike conventional numerical methods, where the pollution effect requires mesh refinement proportional to $k^{3/2}$ or higher, the explicit series reaches arbitrary accuracy by increasing $N$ linearly with $k$ in any dimension $d\geq 2$.

\subsection{Spectral verification}

To validate Theorem~\ref{thm:spectral-basic}, we compute the resonant wavenumbers $k_{\ell,p}(\alpha)$ directly from the characteristic equation \eqref{eq:char-unified} using Brent's root-finding algorithm, and compare with the asymptotic bounds \eqref{eq:bounds}.\smallskip

Figure~\ref{fig:spectral-disk} shows the spectral branches $k_{\ell,p}(\alpha)$ as functions of the impedance $\alpha$ for representative modes $(\ell,p)$ in the unit disk ($d=2$).  In each panel, the solid blue curve is the numerically computed Robin eigenvalue, the green dash-dot line marks the Neumann limit $\mu_{\ell,p}'/R$ (as $\alpha\to0^{+}$), and the red dashed line marks the Dirichlet limit $\mu_{\ell,p}/R$ (as $\alpha\to+\infty$).  For every mode, the branch increases strictly with $\alpha$, interpolating continuously between the two extremes without crossing. The same monotonic interpolation behaviour is observed for the unit ball and all higher dimensions $d\geq3$ (not shown).

\begin{figure}[htbp]
\centering
\includegraphics[width=0.9\textwidth]{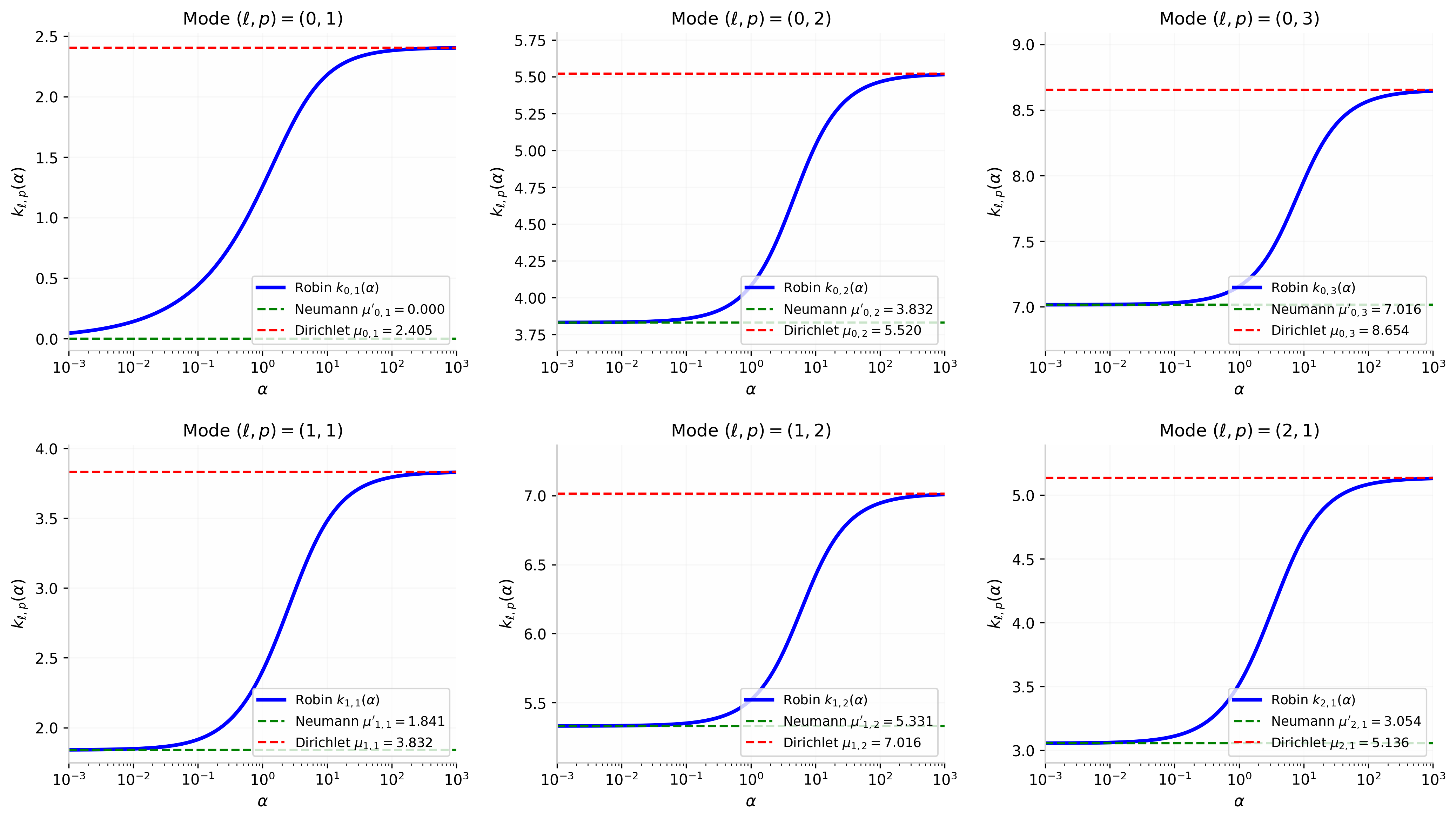}
\caption{Spectral branches $k_{\ell,p}(\alpha)$ for the unit disk ($d=2$).  Solid curves: numerically computed roots of $\Delta_{\ell}(k)=0$.  Dashed lines: Dirichlet eigenvalues $\mu_{\ell,p}$.  Dash-dot lines: Neumann eigenvalues $\mu_{\ell,p}'$.}
\label{fig:spectral-disk}
\end{figure}

\subsection{A global estimate of the critical wave number}
\label{sec:kc-estimate}

We conclude with a closed-form estimate of the critical wave number
$k_c$ for arbitrary $\alpha R$, the smallest positive root of
\begin{equation}\label{eq:char-l0}
x\,J_{\nu+1}(x)=\beta\, J_\nu(x),
\qquad \nu=\tfrac{d-2}{2},\ \ x=k_cR,\ \ \beta=\alpha R.
\end{equation}
The perturbation series of \eqref{eq:char-l0}
in Theorem~\ref{thm:gap}(iii)
$$x_c^2=d\beta-\frac{d}{d+2}\beta^2+O(\beta^3)$$
is accurate only for small $\beta$
(already $10\%$ off at $\beta=1$ for its leading term).  Interpolating
between the first three terms of this series and the Dirichlet limit
$x_c\to \mu_{0,1}$ ($\beta\to\infty$, where $\mu_{0,1}$ is the first
positive zero of $J_\nu$) gives the two-point rational approximation
\begin{equation}\label{eq:kc-global}
x_c^2\;\approx\;
\frac{d\beta\,(1+r\beta)}
{1+s\beta+\dfrac{dr}{\mu_{0,1}^{\,2}}\,\beta^{2}},
\qquad
r=\frac{d\,\mu_{0,1}^{\,2}}{(d+2)(d+4)\,
\bigl(d(d+2)-\mu_{0,1}^{\,2}\bigr)},
\quad s=r+\frac{1}{d+2},
\end{equation}
whose coefficients are determined exactly, without any numerical
fitting.  Table~\ref{tab:kc-global} shows that its relative error stays below $1.2\%$ on the whole range $\beta\in(0,\infty)$ for $d=3$
(below $0.7\%$ for $d=2$ and below $1.7\%$ for $d=4$), while the leading term $\sqrt{d\beta}$ of the series fails already for $\beta>1$.
\vspace{-0.3cm}

\begin{table}[htbp]
\centering
\caption{Approximation \eqref{eq:kc-global} of $x_c=k_cR$ for $d=3$,
against the exact root of \eqref{eq:char-l0}.}
\label{tab:kc-global}
\vspace{0.3cm}
\begin{tabular}{c c c c c}
\hline
$\beta$ & exact $x_c$ & approximation \eqref{eq:kc-global} & rel.\ error
& leading-term error\\
\hline
$1$   & $1.5708$ & $1.5716$ & $0.05\%$ & $10.3\%$\\
$3$   & $2.2889$ & $2.2987$ & $0.43\%$ & $31.1\%$\\
$5$   & $2.5704$ & $2.5907$ & $0.79\%$ & $50.7\%$\\
$10$  & $2.8363$ & $2.8681$ & $1.12\%$ & $93.1\%$\\
$100$ & $3.1102$ & $3.1210$ & $0.35\%$ & $456.9\%$\\
\hline
\end{tabular}
\end{table}

\subsection{High-frequency spacing}
\label{subsec:spacing-num}

We validate Theorem~\ref{thm:spacing} by computing the exact roots
$k_{\ell,p}(\alpha)$ of the characteristic
equation~\eqref{eq:char-unified} for large $p$ and comparing the
consecutive spacing with the asymptotic
formula~\eqref{eq:spacing}.  Figure~\ref{fig:spacing} shows the
results for the unit disk ($d=2$, $R=1$, $\ell=0$, $\alpha=1$),
where $C_{0,2}(1)=5/8$.

The left panel confirms that the actual spacing (blue dots) rapidly
approaches the Weyl leading term $\pi/R$ (grey dotted line) and is
accurately captured by the second-order correction in~\eqref{eq:spacing}
(red dashed curve).  The right panel shows that the relative error
decays as $O((k_{\ell,p}R)^{-3})$: it drops from
$1.7\times10^{-2}$ at $p=2$ to $6.4\times10^{-8}$ at $p=100$,
matching the black $O(p^{-3})$ reference line.  Even for moderately
large $p$ the second-order correction is sufficient to bring the
spacing to within $10^{-4}$ of the exact value, while the
leading-order term $\pi/R$ alone deviates by roughly $0.3\%$ at
$p=5$ and falls below $0.1\%$ at $p=10$.

The explicit correction
$C_{\ell,d}(\alpha)=\alpha R+\ell-\frac{(2\ell+d)^2-1}{8}$
reveals how the Robin condition interpolates between the Neumann
and Dirichlet extremes.  At $\alpha=0$ the correction reduces to
$C_{\ell,d}(0)=\ell-\frac{(2\ell+d)^2-1}{8}$, which governs the
high-frequency spacing of the Neumann problem; as $\alpha$
increases, the linear term $\alpha R$ shifts the spacing
monotonically toward the Dirichlet regime. Only a single linear
term in $\alpha$ controls the second-order spacing correction.
The explicit form of the characteristic equation makes this
quantitative dependence transparent, offering a concrete
benchmark against which more general asymptotic and variational
results can be tested and calibrated.

\begin{figure}[htbp]
\centering
\includegraphics[width=0.95\textwidth]{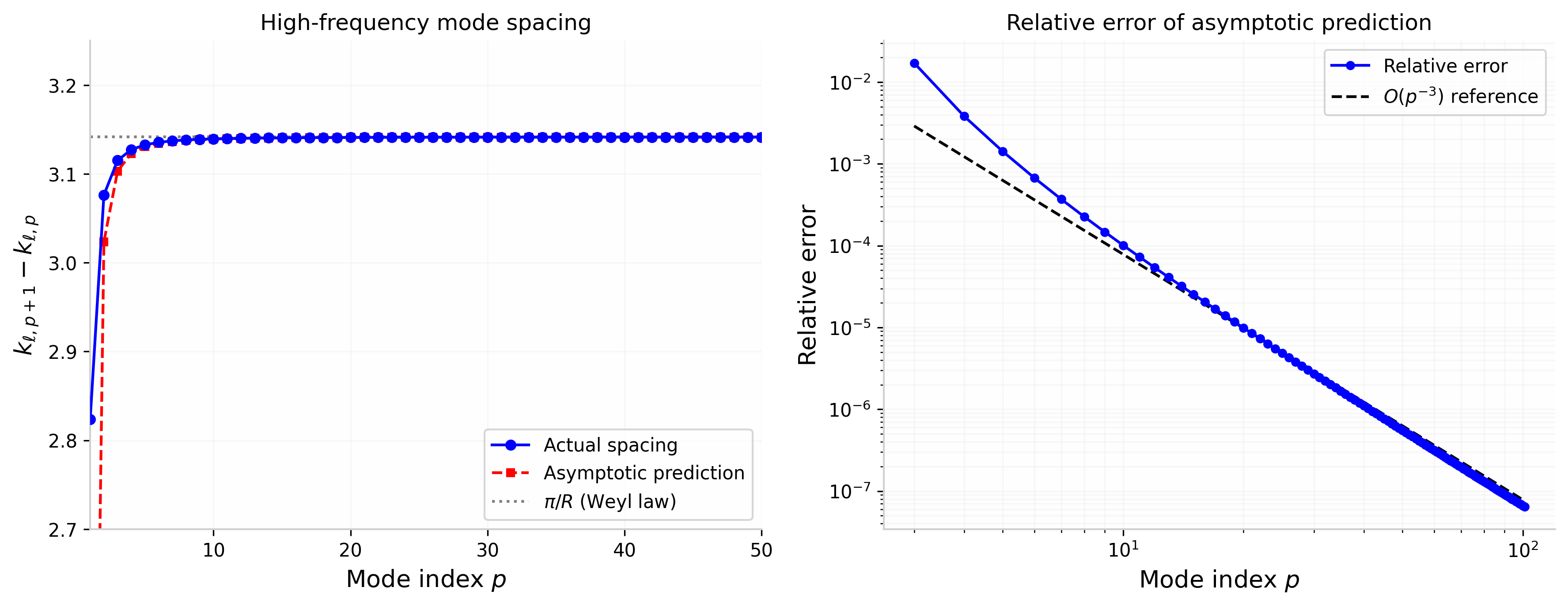}
\caption{High-frequency mode spacing for the unit disk ($d=2$,
$R=1$, $\ell=0$, $\alpha=1$).  Left: consecutive spacing
$k_{0,p+1}-k_{0,p}$ versus mode index $p$.  Right: relative error
of the asymptotic formula~\eqref{eq:spacing}.}
\label{fig:spacing}
\end{figure}

\section{Conclusion}\label{sec:conclusion}

We have obtained the first explicit closed-form Green's functions for the Helmholtz equation with Robin boundary conditions on balls in all dimensions $d\ge 2$. The construction unifies the two-dimensional and higher-dimensional settings through a single algebraic framework: the free-space fundamental solution is expanded on the boundary via Graf's or the hyperspherical addition theorem, and the regular correction is determined mode by mode by matching Robin traces. The resulting series are absolutely convergent on compact interior subsets, symmetric in source and observer, and computable to machine precision with only $N\gtrsim kR$ modes.

These kernels furnish exact forward solutions for impedance-matched cavities, eliminating geometric discretisation error entirely. They serve as non-approximate benchmarks for the validation of finite-element and boundary-element codes on Robin boundaries, and their explicit resonance spectra provide test data for spectral asymptotics and Weyl-law calibration. Because the characteristic determinant is available in closed form, every spectral property---monotonic interpolation, low-frequency gap, and high-frequency spacing correction---follows from elementary Bessel analysis without recourse to abstract variational or Dirichlet-to-Neumann arguments.

The decomposition--expansion method adapts readily to annular and spherical-shell domains, where it leads to small $2\times 2$ linear systems for the modal coefficients. Time-domain analogues via Laplace and Fourier transforms will be treated elsewhere.

\section*{Statements and Declarations}

The author declares that there is no conflict of interest. Data sharing is not applicable to this article as no datasets were generated or analysed during the current study.

\end{document}